\documentclass[12pt]{amsart}
\usepackage{latexsym, amssymb, amsmath}

\newtheorem{conj}{Conjecture}
\newtheorem{thm}{Theorem}[section]

\newtheorem{prop}[thm]{Proposition}
\newtheorem{rem}[thm]{Remark}
\theoremstyle{definition}

\newtheorem{prob}{Problem}

\newcommand{\onabla}{\overline\nabla}

\newcommand{\p}{\phi}

\newcommand{\Ric}{{\rm Ric}}
\newcommand{\grad}{{\rm grad}\,}
\newcommand{\ep}{\varepsilon}

\title{Biharmonic hypersurfaces\\ with bounded mean curvature}

\author{Shun Maeta} 

\thanks{Supported in part by the Grant-in-Aid for Young Scientists(B), No.15K17542, 
  Japan Society for the Promotion of Science. }

\address{\footnotesize{Division of Mathematics, Shimane University, Nishikawatsu 1060 Matsue, 690-8504, Japan. }
 }
\footnotesize{
\email{shun.maeta@gmail.com~{\it or}~maeta@riko.shimane-u.ac.jp}
}

\begin{document} 
\maketitle 
\markboth{Biharmonic hypersurfaces with bounded mean curvature} 
{Shun Maeta}

\begin{abstract} 
We consider a complete biharmonic hypersurface with nowhere zero mean curvature vector field $\phi:(M^m,g)\rightarrow (S^{m+1},h)$ in a sphere. 
If the squared norm of the second fundamental form $B$ is bounded from above by m, and
$\int_M H^{- p }dv_g<\infty$, for some $0<p<\infty$,
then the mean curvature is constant.
\end{abstract}


\qquad\\


\section{Introduction}\label{intro} 

The problem of biharmonic maps was suggested in 1964 by J. Eells and J. H. Sampson (cf. \cite{jejs1}).
Biharmonic maps are generalizations of harmonic maps.
As well known, harmonic maps have been applied into various fields in differential geometry.
However there are non-existence results for harmonic maps.
Therefore a generalization of harmonic maps is an important subject.

G. Y. Jiang \cite{jg1} considered a biharmonic submanifold, and gave some examples of non-minimal biharmonic submanifolds in $S^n$ as follows:
(i) $S^{n-1}(\frac{1}{\sqrt2})$ and 
(ii) $S^{n-p}(\frac{1}{\sqrt2})\times S^{p-1}(\frac{1}{\sqrt2})$, ($n-p\not=p-1$).
\vspace{10pt}

There are many studies of biharmonic submanifolds in spheres.
Interestingly, their studies suggest the following BMO conjecture which was introduced by Balmus, Montaldo and Oniciuc (cf. \cite{21}). 

\vspace{10pt}

\begin{conj}[BMO conjecture]
Any biharmonic submanifold in spheres has constant mean curvature.
\end{conj}

\vspace{10pt}

On the other hand, 
since there is no assumption of {\it completeness} for submanifolds in BMO conjecture, 
in a sense it is a problem in {\it local} differential geometry.  
The author reformulated BMO conjecture into a problem 
in {\it global} differential geometry (cf. \cite{sm13}). 

\vspace{10pt}

\begin{conj} 
Any {\rm complete} biharmonic submanifold in spheres has constant mean curvature.
\end{conj} 

\begin{rem}
Interestingly, Z.-P. Wang and Y.-L. Ou treated a biharmonic Riemannian submersion from a sphere and got non-existence results $($cf.~\cite{zpwylo1}$)$.
\end{rem}

There are affirmative partial answers to BMO conjecture, if $M$ is one of the following:

(i) A compact hypersurface with nowhere zero mean curvature vector field and $|B|^2\geq m$ or $|B|^2\leq m$, where $|B|^2$ is the squared norm of the second fundamental form (cf. \cite{48},\  \cite{16}).
 
(ii) An orientable Dupin hypersurface (cf. \cite{16}).

(iii) A compact submanifold with $|{\bf H}|\geq1$ (cf. \cite{23},~see also \cite{sm13}).

(iv) A complete submanifold with $|{\bf H}|\geq1$ and the Ricci curvature of $M$ is bounded from below (cf. \cite{sm13}).

\vspace{10pt}

In \cite{sm13}, the author showed the following.
\begin{thm}[\cite{sm13}]\label{old main th}
Let $\p:(M^m,g)\rightarrow (S^{m+1},h)$ be a complete biharmonic hypersurface in a sphere.
If the mean curvature $H\geq 1$, and
$$\int_M \left(H^2- 1\right)^p dv_g<\infty,$$
for some $0<p<\infty$, 
 then $H$ is $1$.
\end{thm}

Here we remark that the author obtained some affirmative partial answers to BMO conjecture under more general situation.
 ~Since we gave an affirmative partial answer to BMO conjecture under the assumption $H\geq 1$ in Theorem \ref{old main th}, in this paper, we consider $0<H\leq1.$

Before proving our main theorem, we show the following theorem.

\begin{thm}\label{main lemma}
Let $\p:(M^m,g)\rightarrow (N^{m+1},h)$ be a complete non-positive biminimal hypersurface. 
Assume that the mean curvature $H$ satisfies $0<H\leq1$.
 We also assume that $|B|^2\leq \Ric^N(\xi,\xi)$, where $B$ is the second fundamental form of $M$ in $N$, $\Ric^N$ is the Ricci curvature of N, and $\xi$ is the unit normal vector field on $M$.
 ~If 
 $$\int_MH^{-p}dv_g<\infty,$$
  for some $0<p<\infty$, then $H$ is constant. 
\end{thm}

\begin{rem}
If we assume $\int_M H^{-p}dv_g<\infty$ and $\int_M H^{-(p+\ep)}dv_g<\infty$, for some $\ep>0$ and $0<p<\infty$, then we don't need $H\leq1$.
\end{rem}

By applying Theorem \ref{main lemma}, we can show our main theorem:

\begin{thm}\label{main th}
Let $\p:(M^m,g)\rightarrow (S^{m+1},h)$ be a complete biharmonic hypersurface with nowhere zero mean curvature vector field in a sphere. 
If $|B|^2\leq m$, and
$$\int_MH^{-p}dv_g<\infty,$$
 for some $0<p<\infty$, then $H$ is constant. 
\end{thm}

Theorem $\ref{main th}$ is an affirmative partial answer to BMO conjecture.

In this paper, we assume that the mean curvature vector field is nowhere zero.
The remaining sections are organized as follows. 
Section~$\ref{Pre}$ contains some necessary definitions and preliminary geometric results.
 In section~$\ref{main theorem}$, we prove Theorem $\ref{main lemma}$ and Theorem $\ref{main th}$.

\qquad\\

\section{Preliminaries}\label{Pre} 

In this section, we shall give the definitions of biharmonic hypersurfaces and biminimal hypersurfaces.
\vspace{10pt}

The problem of biharmonic maps was suggested in 1964 by J. Eells and J. H. Sampson (cf. \cite{jejs1},~ \cite{jell1}). 
Biharmonic maps are critical points of the bi-energy functional 
$$E_2 (\phi )=\frac{1}{2}\int_M |\tau (\phi)| ^2 dv_g,$$
on the space of smooth maps $\p:(M^m,g)\rightarrow (N^n,h)$ between two Riemannian manifolds $(M^m,g)$ and $(N^n,h)$.
$\nabla$ and $\nabla^N$ denote the Levi-Civita connections on $(M,g)$ and $(N,h)$, respectively.
 $\onabla$ denotes the induced connection on $\p^{-1}TN$.
In 1986, G. Y. Jiang \cite{jg1} derived the first and the second variational formulas of the bi-energy and studied biharmonic maps.
The Euler-Lagrange equation of $E_2$ is 
\begin{equation}\label{NSbi}
\tau_2(\phi)=-\Delta^{\phi} \tau (\phi ) -\sum^m_{i=1} R^N (\tau (\phi )  , d\phi (e_i))d\phi (e_i)=0,
\end{equation}
 where $\{e_i\}_{i=1}^m$ is an orthonormal frame field on $M$, $\Delta^{\phi}:=\displaystyle\sum^m_{i=1}\left(\onabla_{e_i}\onabla_{e_i}-\onabla_{\nabla_{e_i}e_i}\right)$, $\tau(\p)={\rm Trace}\,\nabla d\p$ is the {\em tension field} and $R^N$ is the Riemannian curvature tensor of $(N,h)$ given by $R^N(X,Y)Z=\nabla^N_X\nabla^N_YZ-\nabla^N_Y\nabla^N_XZ-\nabla^N_{[X,Y]}Z$ for $X,\ Y,\ Z\in \frak{X}(N)$.
$\tau_2(\p)$ is called the {\em bi-tension field} of $\p$.
 A map $\p:(M,g)\rightarrow (N,h)$ is called a {\em biharmonic map} if $\tau_2(\p)=0$. 

\vspace{10pt}

Let $M$ be an $m$-dimensional immersed submanifold in $(N^{m+1},h)$, $\p:(M^m,g)\rightarrow (N^{m+1},h)$ its  immersion and $g$ its induced Riemannian metric.
The Gauss and Weingarten formulas are given by
\begin{equation}
\nabla^N_XY=\nabla _XY+B(X,Y),\ \ \ \ X,Y\in \frak{X}(M),
\end{equation}
\begin{equation}\label{2.Wformula}
\nabla^N_X \xi =-A_{\xi}X,\ \ \ X\in \frak{X}(M),\  \xi \in \frak{X}(M)^{\perp},  
\end{equation}
where $B$ is the second fundamental form of $M$ in $N$, $A_{\xi}$ is the shape operator for a unit normal vector field $\xi$ on $M$.
It is well known that $B$ and $A$ are related by
\begin{equation}\label{2.BA rel}
\langle B(X,Y), \xi \rangle=\langle A_{\xi}X,Y \rangle.
\end{equation}

For any $x \in M$, let $\{e_1, \cdots, e_m,\xi\}$ be an orthonormal basis of $N$ at $x$ such that $\{e_1, \cdots, e_m\}$ is an orthonormal basis of $T_xM$. 
The mean curvature vector field ${\bf H}$ of $M$ at $x$ is given by 
$$ 
{\bf H}(x) = \frac{1}{m} \sum_{i = 1}^m B(e_i, e_i).
$$ 

If an isometric immersion $\p:(M^m,g)\rightarrow (N^{m+1},h)$ is biharmonic, then $M$ is called a {\em biharmonic hypersurface} in $N$.
 In this case, we remark that the tension field $\tau(\p)$ of $\p$ is written as $\tau(\p)=m{\bf H}$.
The necessary and sufficient condition for $M$ in $N$ to be biharmonic is the following: 
\begin{equation}\label{NS bih sub}
\Delta^{\p}{\bf H}+\sum_{i=1}^mR^N({\bf H},d\p(e_i))d\p(e_i)=0.
\end{equation}

From $(\ref{NS bih sub})$, the necessary and sufficient condition for $\p:(M^m,g)\rightarrow (N^{m+1},h)$ to be a biharmonic hypersurface is as follows (cf. \cite{ylo2010}):  

\begin{align}\label{NS bih hyp 1}
\Delta H-H\,|A|^2 +H\,\Ric ^N(\xi,\xi)=0,
\end{align}
\begin{align}\label{NS bih hyp 2}
2A(\grad H)+
\frac{1}{2}m\, \grad H^2- 2H(\Ric^N(\xi))^T=0.
\end{align}

\begin{rem}
Biharmonic hypersurfaces satisfy an overdetermined problem $($see \cite{KU2014}$)$.
\end{rem}

If an isometric immersion $\p:(M^m,g)\rightarrow (N^{m+1},h)$ satisfies 
\begin{align}\label{NS bimini hyp}
\Delta H-H\,|A|^2 +H\,\Ric ^N(\xi,\xi)=\lambda H\ \ \ (\text{for some}~\lambda\in \mathbb{R}),
\end{align}
then $M$ is called a {\em biminimal hypersurface}.
Biminimal hypersurfaces were introduced by E. Loubeau and S. Montaldo (cf. \cite{elsm1}).
We call an biminimal hypersurface {\em free biminimal} if it satisfies the biminimal condition for $\lambda=0$.
 If $M$ is a biminimal hypersurface with $\lambda \leq0$ in $N$, then $M$ is called a {\em non-positive biminimal hypersurface} in $N$.

\vspace{10pt}

\begin{rem}\label{biharmonic is biminimal}
We remark that {\bf every biharmonic hypersurface is free biminimal}.
\end{rem}

\vspace{10pt}

\section{Proof of Theorem $\ref{main lemma}$ and Theorem $\ref{main th}$}\label{main theorem}
\
In this section, we will prove our main theorem.
To prove our main theorem, we will use Petersen-Wylie's Yau-Naber Liouville theorem (cf. \cite{Petersen-Wylie2010}).
Liouville type theorem is a strong tool for biharmonic submanifolds (cf. \cite{sm12},~\cite{Luo2014}).

\begin{thm}[\cite{Petersen-Wylie2010}]
Let $(M,g)$ be a manifold with finite $h$-volume: $\int_Me^{-h}dv_g<\infty$.
If $u$ is a smooth function in $L^2(e^{-h}dv_g)$ which is bounded below such that $\Delta_h u\geq0,~(\Delta_h=\Delta-\nabla_{\nabla h})$, then $u$ is constant. 
\end{thm}

\vspace{10pt}

We prove Theorem $\ref{main lemma}$.
\begin{proof}[Proof of Theorem $\ref{main lemma}$]

Let $\ep>0$ be small enough. One can easily compute 
\begin{align}\label{pr main 1}
\Delta H^{-\ep}
&=\ep(\ep+1)H^{-(\ep+2)}|\nabla H|^2-\ep H^{-(\ep+1)}\Delta H \\
&=\ep(\ep+1)H^{-(\ep+2)}|\nabla H|^2-\ep H^{-\ep}|A|^2+\ep H^{-\ep}\Ric^N(\xi,\xi)-\lambda \ep H^{-\ep},\notag
\end{align}
and
\begin{align}
\nabla_{\nabla h}H^{-\ep}=-\ep H^{-(\ep+1)}\langle \nabla h,\nabla H\rangle,
\end{align}
where the second line of $(\ref{pr main 1})$, we used $(\ref{NS bimini hyp})$.
Thus we have 
\begin{align}\label{Del ep}
\Delta_hH^{-\ep}
=&\ep(\ep+1)H^{-(\ep+2)}|\nabla H|^2-\ep H^{-\ep}|A|^2+\ep H^{-\ep}\Ric^N(\xi,\xi)\\
&-\lambda \ep H^{-\ep}+\ep H^{-(\ep+1)}\langle \nabla h,\nabla H\rangle.\notag
\end{align}
Set $h=\log H^{(p-1)}$. Since we have 
$$\nabla h=(p-1)\frac{\nabla H}{H},$$
one can obtain that 
\begin{align*}
\ep H^{-(\ep+2)}
&\left\{
(\ep+1)|\nabla H|^2
+H\langle \nabla h,\nabla H\rangle
\right\}\\
&=\ep (\ep+p)H^{-(\ep+2)}
|\nabla H|^2\geq0.
\end{align*}
On the other hand, by assumption,
\begin{align*}
\ep H^{-\ep}(&-|A|^2+\Ric^N(\xi,\xi)-\lambda)\\
&\geq\ep H^{-\ep}(-|A|^2+\Ric^N(\xi,\xi))
\geq0,
\end{align*}
where we used $|B|^2=|A|^2$.
Therefore we obtain $\Delta_h H^{-\ep}\geq0.$

Since $h=\log H^{(p-1)}$, by assumption, we have
$$\int_Me^{-h}dv_g=\int_MH^{-(p-1)}dv_g\leq\int_M H^{-p}dv_g<\infty.$$
On the other hand, one can get that
$$\int_M H^{-\ep}e^{-h}dv_g=\int_M H^{-(p-1+\ep)}dv_g\leq\int_M H^{-p}dv_g<\infty.$$
Applying Petersen-Wylie's Yau-Naber Liouville theorem, we obtain $H^{-\ep}$ is constant.
Therefore $H$ is constant.
\end{proof}

Applying Theorem $\ref{main lemma}$, one can prove our main theorem (Theorem~$\ref{main th}$).

\begin{proof}[Proof of Theorem $\ref{main th}$]
Since $N=S^{m+1},$ $Ric^N(\xi,\xi)=m$. By assumption, one can obtain $|B|^2\leq m=\Ric^N(\xi,\xi)$.
 Since $mH^2\leq |B|^2$, $H\leq 1$ is automatically satisfied.
 Note that biharmonic hypersurfaces are non-negative biminimal.
Applying Theorem $\ref{main lemma}$, we obtain $H$ is constant.
\end{proof}

\quad\\
\quad\\

\section{Appendix}\label{Appendix}

We can apply our method to $p$-biharmonic submanifolds (cf. \cite{Hornung-Moser2014}).
If an isometric immersion $\p:(M,g)\rightarrow (N,h)$ satisfies
$$\Delta^{\p}(|{\bf H}|^{p-2}{\bf H})+R^N(|{\bf H}|^{p-2}{\bf H},d\p(e_i))d\p(e_i)=0,$$
then $M$ is called a {\em $p$-biharmonic submanifold}.
For $p$-biharmonic submanifolds, it is easy to see that we can get same (similar) results as in the results of biharmonic submanifolds in many cases.
(For example, Corollary $3.6$, $3.9$ in \cite{sm12}, and so on.)
In fact, the same argument as in Proof of Theorem $\ref{main lemma}$
shows the following result.
\begin{prop}
Let $\p:(M^m,g)\rightarrow (N^{m+1},h)$ be a complete $p$-biharmonic hypersurface. 
Assume that the mean curvature $H$ satisfies $0<H\leq1$.
 We also assume that $|B|^2\leq \Ric^N(\xi,\xi)$. If 
 $$\int_MH^{-q}dv_g<\infty,$$
  for some $0<q<\infty$, then $H$ is constant. 
\end{prop}

\begin{proof}
Set $u=H^{p-1}$. We have only to consider $\Delta u^{-\ep}$ and $h=\log u^{\frac{q}{p-1}-\ep}$.
\end{proof}

Therefore we give one problem.

\begin{prob}
Does any (complete) $p$-biharmonic submanifold in spheres have constant mean curvature?
\end{prob}
\qquad\\





\bibliographystyle{amsbook}

\end{document}